\documentclass[11pt,oneside,a4paper]{amsart}
\usepackage{microtype}
\usepackage{amsmath,amssymb}
\usepackage{hyperref}

\usepackage{newtxtext,newtxmath}
\renewcommand{\epsilon}{\varepsilon}
\renewcommand{\phi}{\varphi}

\newtheorem{corollary}{Corollary}
\newtheorem{lemma}{Lemma}
\newtheorem{theorem}{Theorem}

\makeatletter
\renewcommand\subsection{\@startsection{subsection}{2}%
  \z@{.5\linespacing\@plus.7\linespacing}{.1\linespacing}%
  {\itshape}}
\makeatother

\title[Strict convexity for GJE in dimension two]{Strict convexity and $C^1$ regularity of solutions to generated Jacobian equations in dimension two}
\author[C. Rankin]{Cale Rankin}
\address{Australian National University}
\email{cale.rankin@anu.edu.au}
\thanks{This research is supported by an Australian Government Research Training Program (RTP) Scholarship.}
\begin{document}

\maketitle

\begin{abstract}
  We present a proof of strict $g$-convexity in 2D for solutions of generated Jacobian equations with a $g$-Monge--Amp\`ere measure bounded away from 0. Subsequently this implies $C^1$ differentiability in the case of a $g$-Monge--Amp\`ere measure bounded from above. Our proof follows one given by Trudinger and Wang in the Monge--Amp\`ere case. Thus, like theirs, our argument is local and yields a quantitative estimate on the $g$-convexity. As a result our differentiability result is new even in the optimal transport case: we weaken previously required domain convexity conditions.  Moreover in the optimal transport case and the Monge--Amp\`ere case our key assumptions, namely A3w and domain convexity,  are necessary.  
\end{abstract}

\section{Introduction}
Generated Jacobian equations are PDEs of the form
\begin{equation}
  \label{eq:gje}
   \det DY(\cdot,u,Du) = \psi(\cdot,u,Du),
\end{equation}
where the vector field $Y$ has a particular structure. This class of equations includes the Monge--Amp\`ere equation and the Jacobian equation from optimal transport as special cases. Precise statements concerning the structure of $Y$ are given in Section \ref{sec:genfunc}. For now we state our purpose is to show, for generalised notions of convexity, that solutions of
\begin{align}
  \label{eq:ineq}
   \det DY(\cdot,u,Du) \geq c > 0 
\end{align}
on a convex domain $\Omega \subset\mathbf{R}^2$ are strictly convex. By considering an analogue of the Legendre transform we obtain that when \eqref{eq:ineq} is instead a bound from above and $u$ satisfies the  second boundary value problem
\begin{align}
  \label{eq:2bvp}
  Y(\cdot,u,Du)(\Omega) = \Omega^*,
\end{align}
 for a domain $\Omega\subset\mathbf{R}^2$ and a convex domain $\Omega^*\subset\mathbf{R}^2$, then $u$ is $C^1$.

 The class of $Y$ we work with, those obtained from a generating function, embraces applications in optimal transport and geometric optics, whilst allowing the use of a large family of techniques developed for the study of the Monge--Amp\`ere equation. Thus these equations, first studied by Trudinger \cite{Trudinger14}, provide a good combination of applicability and tractability. For example, our proof follows the corresponding proof for the Monge--Amp\`ere equation as given by Trudinger and Wang \cite[Remark 3.2]{TrudingerWang08}. Though there the result is originally due to Alexandrov \cite{Alexandroff1942} and Heinz \cite{Heinz1959}.

 The work most relevant to ours is that of Figalli and Loeper \cite{FigalliLoeper09} who dealt with optimal transport, that is when $Y = Y(\cdot,Du)$. They proved the $C^1$ differentiability of solutions in two dimensions under a bound from above on the $c$-Monge--Amp\`ere measure. They assumed a uniformly $c$-convex source $\Omega$ and a strictly $c^*$-convex target $\Omega^*$. Thanks largely to the maturity of the relevant convexity theory we are able to reduce these convexity conditions. Our $C^1$ result requires no convexity condition on the source and only $g^*$-convexity on the target. This condition is necessary even for the Monge--Amp\`ere equation. 

 Our results are inherently two dimensional --- in higher dimensions strict convexity does not hold under only a bound from below.  However one of the key applications of generated Jacobian equations (GJE) is geometric optics and this takes place in two dimensions.  In higher dimensions a two sided bound on the Monge--Amp\`ere measure is required for strict convexity and differentiability. Here the relevant work is that of Caffarelli \cite{Caffarelli90} for the Monge--Ampere equation, Chen, Wang \cite{ChenWang16}, Figalli, Kim, McCann \cite{FKM13}, Guillen, Kitagawa \cite{GuillenKitagawa15}, and V\'etois \cite{Vetois15} for optimal transport, and Guillen, Kitagawa \cite{GuillenKitagawa17} for GJE.

 Our plan is to define generating functions, GJE, and the related convexity notions in Section \ref{sec:genfunc}. Here we also state precisely our main results: Theorem \ref{thm:main} and Corollary \ref{cor:c1}. In Section \ref{sec:main-lemma} we state a versatile differential inequality (essentially taken from \cite{Trudinger14,Trudinger20}) whose proof we relegate to the appendix. We prove strict convexity in Section \ref{sec:strict-convexity-2d} and conclude with $C^1$ differentiability in Section \ref{sec:c1}.

 \subsection*{Acknowledgements}
Thanks to Neil Trudinger who suggested extending the proof in \cite{TrudingerWang08} to generated Jacobian equations.

\section{Generating functions and GJE}
\label{sec:genfunc}
Generated Jacobian equations are equations of the form \eqref{eq:gje} where $Y$ derives from a generating function (defined below). This requirement allows us to develop a framework extending convexity theory. The material below is largely due to Trudinger \cite{Trudinger14} with other presentations in \cite{GuillenKitagawa17,Jeong20,Jhaveri17,JiangTrudinger18,LiuTrudinger16,Rankin20}.  Guillen's survey article \cite{Guillen19} is a good introduction and lists the 2D theory we develop here as an open problem. Throughout it is helpful to keep in mind the cases $g(x,y,z) = x\cdot y-z$ and  $g(x,y,z) = c(x,y)-z$ where $c$ is a cost function from optimal transport. In these settings $g$-convexity  simplifies to standard convexity and the cost convexity of optimal transport respectively.

\subsection{Generating functions}
\label{sec:generating-functions}
The structure of a particular generated Jacobian equation derives from a generating function. We denote this function by $g$, and require it satisfy the following assumptions.\\
  \textbf{A0.} $g \in C^4(\overline{\Gamma})$ where $\Gamma \subset \mathbf{R}^n\times\mathbf{R}^n\times\mathbf{R}$ is a bounded domain satisfying that the projections 
  \[  I_{x,y}:=\{z ; (x,y,z) \in \Gamma\}\]
are open intervals. Moreover we assume there are domains $\Omega,\Omega^* \subset \mathbf{R}^n$  such that for all $x \in \overline{\Omega}, y \in \overline{\Omega^*}$ we have that $ I_{x,y}$ is nonempty.\\
  \textbf{A1.} For all $(x,u,p) \in \mathcal{U}$ defined as
  \[\mathcal{U} = \{(x,g(x,y,z),g_x(x,y,z)); (x,y,z) \in \Gamma\},\]
  there is a unique $(x,y,z)\in\Gamma$ such that
\begin{align*}
  g(x,y,z) = u \quad \text{ and }\quad g_x(x,y,z) = p.
\end{align*}
\textbf{A2. } On $\overline{\Gamma}$ there holds $g_z<0$ and the matrix E with entries
\[ E_{ij}:= g_{x_i,y_j}-(g_z)^{-1}g_{x_i,z}g_{y_j},\]
satisfies $\det E \neq 0$.

\subsection{Generated Jacobian equations}
\label{sec:gener-jacob-equat}
Assumption A1 allows us to define mappings $Y:\mathcal{U}\rightarrow\mathbf{R}^n$ and $Z:\mathcal{U} \rightarrow \mathbf{R}$ by requiring they solve
\begin{align}
 \label{eq:yzdef1} g(x,Y(x,u,p),Z(x,u,p)) &= u,\\
 \label{eq:yzdef2} g_x(x,Y(x,u,p),Z(x,u,p)) &= p.
\end{align}
We call a PDE of the form
  \begin{equation}
    \label{eq:3gje}
\tag{1} \det DY(\cdot,u,Du) = \psi(\cdot,u,Du),
  \end{equation}
  a \textit{generated Jacobian equation} provided  $Y$ derives from solving \eqref{eq:yzdef1} and \eqref{eq:yzdef2} for some generating function.

  Generated Jacobian equations may be rewritten as Monge--Amp\`ere type equations. Suppose $u \in C^2(\Omega)$ satisfies \eqref{eq:3gje}. Then differentiating \eqref{eq:yzdef1} and \eqref{eq:yzdef2} evaluated at $(x,Y(x,u,Du),Z(x,u,Du))$ yields
\begin{align*}
  Du &= g_x+g_yDY+g_zDZ,\\
 \text{ and }\quad D^2u &= g_{xx}+g_{xy}DY+g_{xz}DZ.
\end{align*}
Since the first equation with \eqref{eq:yzdef2} implies
\begin{equation}
  \label{eq:matederive1}
  D_jZ = -\frac{1}{g_z}g_{y_k}D_jY^k,
\end{equation}
the second can be solved for $DY$ by which we have
\begin{align}
\label{eq:yallderiv}  DY(\cdot,u,Du) = E^{-1}[D^2u-g_{xx}(Y(\cdot,u,Du),Z(\cdot,u,Du))],
\end{align}
with $E$ as in A2. Thus $C^2$ solutions of \eqref{eq:gje}  solve
\begin{equation}
  \label{eq:mate}
   \det [D^2u-A(\cdot,u,Du)] = B(\cdot,u,Du),
\end{equation}
for
\begin{align*}
  A(\cdot,u,Du) &= g_{xx}(Y(\cdot,u,Du),Z(\cdot,u,Du)),\\
  B(\cdot,u,Du) &= \det E(\cdot,u,Du)\psi(\cdot,u,Du).
\end{align*}
This PDE is elliptic when $D^2u > g_{xx}(Y(\cdot,u,Du),Z(\cdot,u,Du))$ as matrices. A necessary condition for ellipticity is $B>0$.

\subsection{$g$-convex functions}
\label{sec:g-convex-functions}
We introduce an analogue of convexity theory in which the generating function plays the role of a supporting hyperplane. We say a function $u:\Omega  \rightarrow \mathbf{R}$ is $g$\textit{-convex} provided for all $x_0\in \Omega$ there is $y_0,z_0$ such that
  \begin{align}
\label{eq:3gconvdef1}    g(x_0,y_0,z_0) &= u(x_0),\\
 \label{eq:3gconvdef2}   g(x,y_0,z_0) &\leq u(x) \quad \text{ for all } x\neq x_0,x \in \Omega.
  \end{align}
In this case $g(\cdot,y_0,z_0)$ is called a $g$\textit{-support} at $x_0$. We say $u$ is \textit{strictly} $g$\textit{-convex} if the inequality in \eqref{eq:3gconvdef2} is strict, that is, any given $g$-support only touches $u$ at a single point. Note when $u$ is differentiable, since $u(\cdot)-g(\cdot,y_0,z_0)$ has a minimum at $x_0$, we have $Du(x_0) = g_x(x_0,y_0,z_0)$. This combined with \eqref{eq:3gconvdef1} is equivalent to \eqref{eq:yzdef2} and \eqref{eq:yzdef1} so that $y_0 = Y(x_0,u(x_0),Du(x_0))$ and $z_0 = Z(x_0,u(x_0),Du(x_0))$. Moreover, if $u$ is $C^2$, again since $u(\cdot)-g(\cdot,y_0,z_0)$ has a minimum at $x_0$, we have $D^2u(x_0) \geq  g_{xx}(x_0,y_0,z_0)$ and \eqref{eq:mate} is degenerate elliptic.  However when $u$ is not differentiable \eqref{eq:3gconvdef1} and \eqref{eq:3gconvdef2} may hold for more than one $y_0,z_0$. The set of $y_0$ for which there is $z_0$ such that \eqref{eq:3gconvdef1} and \eqref{eq:3gconvdef2} hold is denoted $Y_u(x_0)$. Similarly for $Z_u(x_0)$. When both are a singleton we identify these sets with their single element.

\subsection{Alexandrov solutions}
\label{sec:aleks-solut}
The mapping $Y_u$ allows us to define a notion of Alexandrov solution. A $g$-convex function $u$ is called an \textit{Alexandrov solution} of \eqref{eq:3gje} on $\Omega$ provided for every Borel $E \subset \Omega$
\[ |Y_u(E)| = \int_{E} \psi(x,u,Du) \ dx.\]
Since $g$-convex functions are locally semi-convex and thus differentiable almost everywhere the integrand on the right-hand side is well defined almost everywhere. We see, via the change of variables formula, that when $u$ is $C^2$ and $Y(\cdot,u,Du)$ is a $C^1$ diffeomorphism Alexandrov solutions are classical solutions. Moreover in the case of inequality \eqref{eq:ineq} the notion of Alexandrov solution is that for every $E\subset \Omega$
\[ |Y_u(E)| \geq c|E|.\]
Here we call the measure $\mu$ defined on Borel $E \subset \Omega$ by $\mu(E) = |Y_u(E)|$ the $g$\textit{-Monge--Amp\`ere measure}.

\subsection{Domain Convexity}
\label{sec:domain-convexity}
With the goal of introducing a notion of domain convexity we introduce a generalisation of line segments. A collection of points $\{x_\theta\}_{\theta\in[0,1]}$ is called a \textit{$g$-segment} with respect to $y\in\Omega^*,z\in \cap_{\theta}I_{x_\theta,y}$  joining $x_0$ to $x_1$ provided
\[\frac{g_y}{g_z}(x_\theta,y,z) = \theta \frac{g_y}{g_z}(x_1,y,z) + (1-\theta)\frac{g_y}{g_z}(x_0,y,z). \]
Using this we say a set $\Omega$ is \textit{$g$-convex} with respect to $(y,z)$ provided for each $x_0,x_1\in\Omega$ the above $g$-segment is contained in $\Omega$. The $g$-segment is unique via condition A1$^*$ in Section \ref{sec:dual-gener-funct}.  If $A\subset \mathbf{R}^n$ and $B\subset\mathbf{R}$ we say $\Omega$ is $g$-convex with respect to $A\times B$ if $\Omega$ is $g$-convex with respect to every $(y,z) \in A\times B.$

\subsection{Conditions for regularity}
\label{sec:cond-regul}
Domain convexity and a Monge--Amp\`ere measure bounded below is necessary and sufficient for strict convexity in 2D in the Monge--Amp\`ere case. In the optimal transport case we need the now well known A3w condition. This condition was introduced for optimal transport by Ma, Trudinger, and Wang \cite{MTW05} and generalised to GJE by Trudinger. For GJE we use an additional condition A4w (due to Trudinger \cite{Trudinger14}).

\textbf{A3w.} A generating function $g$ is said to satisfy the condition A3w provided for all $\xi,\eta\in\mathbf{R}^n$ with $\xi \cdot \eta = 0$ and $(x,u,p) \in \mathcal{U}$ there holds
\begin{equation}
  \label{eq:A3}
  D_{p_kp_l}A_{ij}(x,u,p)\xi_i\xi_j\eta_k\eta_l \geq 0.
\end{equation}
\textbf{A4w.} The matrix $A$ is non-decreasing in $u$, in the sense that for any $\xi \in \mathbf{R}^n$
\[ D_uA_{ij}(x,u,p)\xi_i\xi_j \geq 0.\]

In light of Lemma \ref{lem:diffineq} (\S \ref{sec:main-lemma}) we regard A3w and A4w as tools for controlling how $g$-convex functions separate from their supporting hyperplanes. 

\subsection{The dual generating function}
\label{sec:dual-gener-funct}
Strictly convex functions have $C^1$ Legendre transforms. This provides a useful technique for proving solutions of Monge--Amp\`ere inequalities are $C^1$. The same technique in the $g$-convex case requires the $g^*$-transform which is defined in terms of the dual generating function. We introduce the dual generating function here and the $g^*$ transform in Section \ref{sec:c1}. We set
\[\Gamma^* = \{(x,y,g(x,y,z)); (x,y,z) \in \Gamma\}.\] The \textit{dual generating function}, $g^*$, is the unique function defined on $\Gamma^*$ by
\begin{align}\label{eq:todiff}
  g(x,y,g^*(x,y,u)) = u.
\end{align}
It follows that if $(x,y,z) \in \Gamma$ then
\begin{align}\label{eq:gtodiff2}
  g^*(x,y,g(x,y,z)) = z.
\end{align}
Further, if $u$ is $g$-convex and $y \in Y_u(x_0)$ then the corresponding support is $g(\cdot,y_0,g^*(x_0,y_0,u(x_0)))$.

We introduce dual conditions on $g^*$.\\
\textbf{A1$^{*}$.} For all $(y,z,q) \in \mathcal{V}$ defined as
\[\mathcal{V} = \{(y,g^*(x,y,u),g_y^*(x,y,u));(x,y,u) \in \Gamma^*\},\]
there is a unique $(x,y,u) \in \Gamma^*$ such that
\begin{align*}
  g^*(x,y,u) = z \quad \text{ and } \quad g_y^*(x,y,u) = q.
\end{align*}
Moreover we define mapping $X(y,z,q),U(y,z,q)$ as the unique $x,u$ that satisfy these equations (cf. \eqref{eq:yzdef1} and \eqref{eq:yzdef2}). As remarked in Section \ref{sec:domain-convexity} A1$^*$ implies uniqueness of the $g$-segment between two points. For this note by differentiating \eqref{eq:gtodiff2} we obtain $\frac{g_y}{g_z}(x,y,z) = -g^*_y(x,y,g(x,y,z))$. The right hand side is injective in $x$ for fixed $y,z$. 

We define dual objects by swapping the roles of $x$ and $y$ as well as $z$ and $u$. In particular $g^*$-convex functions are those defined on $\Omega^*$ with $g^*$ supports, and for a $g^*$-convex $v$ we have the mappings $X_v(\cdot),U_v(\cdot)$ which are analogues of $Y_u,Z_u$. Similarly $g^*$-segments are used to define $g^*$-convex sets. We also define dual conditions A2$^*$, A3$^*$, A4w$^*$. However Trudinger \cite{Trudinger14} showed the conditions A2$^*$ and A3w$^*$ are satisfied provided A2 and A3w are.

\subsection{Main results}
\label{sec:main-results}

\begin{theorem}\label{thm:main}
  Let $g$ be a generating function satisfying A0, A1, A2, A1$^*$, A3w and A4w. Let $\Omega\subset\mathbf{R}^2$ and $u:\Omega \rightarrow \mathbf{R}$ be a $g$-convex function satisfying for $c>0$ and all $E\subset\Omega$
  \begin{align}
    \label{eq:alekssoln}
       |Y_u(E)| \geq c|E|.
  \end{align}
If $\Omega$ is $g$-convex with respect to $Y_u(\Omega) \times Z_u(\Omega)$  then $u$ is strictly $g$-convex. 
\end{theorem}
\begin{corollary}\label{cor:c1}
  Let $g$ be a generating function satisfying A0, A1, A2, A1$^*$, A3w, and A4w$^*$. Let $\Omega\subset\mathbf{R}^2$ and $u:\Omega \rightarrow \mathbf{R}$ be a $g-$convex function satisfying that for $C>0$ and all $E\subset\Omega$
  \begin{align}
    \label{eq:alekssolnabove}
    |Y_u(E)| \leq C|E|,\\
    Y_u(\Omega) = \Omega^*.
  \end{align}
  If $\Omega^*$ is $g^*$-convex with respect to $\Omega \times u(\Omega)$ then $u \in C^1(\Omega)$. 
\end{corollary}

Note conditions A4w and A4w$^*$ are always satisfied in the optimal transport case. Indeed Corollary \ref{cor:c1} implies the $C^1$ differentiability of potentials for the optimal transport problem (and subsequently continuity of the optimal transport map) whenever the right hand side of the associated Monge--Amp\`ere type equation is bounded from above, the target is $c^*$-convex, and the A3w condition is satisfied. These conditions are all known to be necessary. Necessity of the $c^*$-convexity is due to Ma, Trudinger, and Wang \cite{MTW05} whilst necessity of A3w is due to Loeper \cite{Loeper09}.

The proof of Theorem \ref{thm:main} is given in Section \ref{sec:strict-convexity-2d}. Finally the proof of Corollary \ref{cor:c1} (which follows from Theorem \ref{thm:main}) is given in Section \ref{sec:c1}. 

\section{Main lemma}
\label{sec:main-lemma}
We make frequent use of a differential inequality for the difference of a $g$-convex function and a $g$-affine function. Similar inequalities are used in \cite{Trudinger14,Trudinger20} and the works  of Kim and McCann \cite[Proposition 4.6]{KimMcCann10} and Guillen and Kitagawa \cite[Lemma 9.3]{GuillenKitagawa17}. 
\begin{lemma}\label{lem:diffineq}
  Let $g$ be a generating function satisfying A0, A1 and A2. Let $y_0 \in \Omega^*,z_0 \in \cap _{x \in \Omega}I_{x,y}$ be given, $\{x_\theta\}$ a $g$-segment with respect to $y_0,z_0$, and $u$ a $C^2$ $g$-convex function. Then
  \[ h(\theta) := u(x_\theta) - g(x_\theta,y_0,z_0),\]
  satisfies
\begin{align}
  \label{eq:maindiffineq} \frac{d^2}{d \theta^2}h(\theta) &\geq [D_{ij}u(x_\theta) - g_{ij}(x_\theta,Y_u(x_\theta),Z_u(x_\theta))]\dot{(x_\theta)}_i\dot{(x_\theta)}_j\\
          \nonumber              &\quad+D_{p_kp_l}A_{ij}\dot{(x_\theta)}_i\dot{(x_\theta)}_jD_kh(\theta)D_lh(\theta)\\
\nonumber  &\quad\quad A_{ij,u}\dot{(x_\theta)}_i\dot{(x_\theta)}_j h(\theta)-  K|h'(\theta)|,
\end{align}
where $K$ depends only on the values of $g$ and its derivatives on $(x_\theta,y_0,z_0)$ and $\dot{x_\theta} = \frac{d}{d\theta}x_\theta$. We've used the shorthand $D_kh(\theta) = u_k(x_\theta)-g_k(x_\theta,y_0,z_0)$. The arguments of $A_{ij,u}$ and $D_{p_kp_l}A_{ij}$ are given in the proof.
\end{lemma}

The proof can be found in \cite{Trudinger14}. For completeness we include a proof in  Appendix \ref{app:proof}. To use A3w in \eqref{eq:maindiffineq} we need to control $D_{p_kp_l}A_{ij}\xi_i\xi_j\eta_k\eta_l$ for arbitrary $\xi,\eta$. We claim if $\xi,\eta \in \mathbf{R}^n$ are arbitrary then A3w implies 
\begin{equation}
  \label{eq:A3equiv}
 D_{p_kp_l} A_{ij}\xi_i\xi_j\eta_k\eta_l \geq - K |\xi||\eta||\xi\cdot \eta|.
\end{equation}
Here $K$ is a non-negative constant depending only on $\Vert D_{p^2}A\Vert_{C^0}$. To obtain \eqref{eq:A3equiv} from A3w first prove it for arbitrary unit vectors $\xi,\eta$ by using A3w with $\xi$ as given and $\eta$ replaced by $\eta - \xi\cdot\eta \xi$.

Thus when $A3w,A4w$ are satisfied and $h(\theta) \geq 0$ we have \[h'' \geq - K|Du-g_{x}||\dot{x_\theta}||h'| \geq -C|h'|,\] where $C$ depends on $\sup|Du|$ and $|\dot{x_\theta}|.$  This inequality appears in \cite[Lemma 9.3]{GuillenKitagawa17},\cite[pg. 310, pg. 315]{Villani09} and \cite{Trudinger14,Trudinger20}. It implies, amongst other things, estimates on  $h'$ in terms of $\sup |h|$ via the following lemma. 

\begin{lemma}\label{lem:diffseg}
  Let $h \in C^2([a,b])$ be a function  satisfying $h'' \geq -K|h'|$. For $t \in (a,b)$ there holds
  \[ -C_0 \sup_{[a,t]}|h|\leq h'(t) \leq C_1\sup_{[t,b]}|h|,\]
  where $C_0,C_1$ depend on $t-a,b-t$ respectively and $K$.
\end{lemma}
\begin{proof}
  First note if $h'(\tau) = 0$ at any $\tau \in (a,b)$ then $h'(a) \leq 0$. To see this assume $\tau$ is the infimum of points with $h'(\tau) = 0$. By continuity if $\tau=a$ we are done. Otherwise $h'$ is single signed on $(a,\tau)$ and if $h'<0$ on this interval then again by continuity we are done. Thus we assume $h'>0$ on $(a,\tau)$. The  inequality $h'' \geq -K|h'|$ implies $\frac{d}{dt}\log(h'(t)) \geq -K$ on $(a,\tau)$. Subsequently for $a<t_1<t_2<\tau$ integration gives
  \begin{equation}
    \label{eq:diffint}
        h'(t_1) \leq e^{K(t_2-t_1)}h'(t_2),
  \end{equation}
  and sending $t_1 \rightarrow a, t_2 \rightarrow \tau$ gives $h'(a) \leq 0$.

  Now to prove the inequality in the lemma let's deal with the upper bound first. We assume $h' >0$ on $[t,b]$ otherwise the argument just given implies $h'(t) \leq 0$. We obtain \eqref{eq:diffint} for $t_1=t$ and $t_2 \in (t,b)$. Integrating with respect to $t_2$ from $t$ to $b$ establishes the result. The other inequality follows by applying the same argument to the function $k$ defined by $k(t):=h(-t)$. 
\end{proof}

\section{Strict convexity in 2D}
\label{sec:strict-convexity-2d}
In this section we present the proof of Theorem \ref{thm:main}. The proof follows closely the proof of Trudinger and Wang \cite[Remark 3.2]{TrudingerWang08} who obtained the same result in the Monge--Amp\`ere case. The key ideas of our proof will be more transparent if the reader is familiar with their proof. There are two key steps: First we obtain a quantitative $g$-convexity estimate for $C^2$ solutions of $\det [D^2u-A(\cdot,u,Du)] \geq c$ (importantly the estimate is independent of bounds on second derivatives). Then we obtain a convexity estimate for Alexandrov solutions via a barrier argument. 
\begin{proof}[Theorem \ref{thm:main}.] \textit{Step 1. Quantitative convexity for $C^2$ solutions}\\
  Initially we assume $u$ is $C^2$. Let $g(\cdot,y_0,z_0)$ satisfy $u \geq g(\cdot,y_0,z_0)$ for $y_0 \in Y_u(\Omega),z_0 \in Z_u(\Omega)$. Assume for some $\sigma \geq 0$ there is distinct $x_{-1},x_{1} \in \Omega$ with
  \begin{align}
    \label{eq:sigest1} u(x_{-1}) \leq g(x_{-1},y_0,z_0)+\sigma,\\
    \label{eq:sigest2}    u(x_1) \leq g(x_1,y_0,z_0)+\sigma.
  \end{align}

  Let $\{x_\theta\}_{\theta \in [-1,1]}$ denote the $g$-segment with respect to $y_0,z_0$ that joins $x_{-1}$ to $x_{1}$ and set
\[ h_\sigma(x) = u(x) - g(x,y_0,z_0) - \sigma.\]
  We use the shorthand $h_\sigma(\theta) = h_\sigma(x_\theta)$. Lemma \ref{lem:diffineq} along with A3w and A4w yields $h''_0(\theta) \geq -K |h_0'(\theta)|$ with the same inequality holding for $h_\sigma$. Hence, via the maximum principle, $h_\sigma$ is less than or equal to its value at the end points. Thus
  \begin{align}
  0 \geq h_\sigma(\theta) \geq \inf_{\theta \in [-1,1]} h_\sigma(\theta) =:-H.\label{eq:heightest}
  \end{align}
  The convexity estimate we intend to derive is $H \geq C > 0$ where $C$ depends only on $|x_0-x_1|, \ \Vert u \Vert_{C^1}, \ g, \ c$. We use $C$ to indicate any positive constant depending only on these quantities.  

  Now, via semi-convexity, $g$-convex functions are locally Lipschitz. Thus for $\theta \in [-3/4,3/4]$ and $\xi \in \mathbf{R}^n$ sufficiently small, \eqref{eq:heightest} implies
  \begin{align}
        \label{eq:lipest1}         h_\sigma(x_\theta+\xi)  &\leq h_\sigma(x_\theta) + C|\xi| \leq C|\xi|,\\
    \label{eq:lipest2}         h_\sigma(x_\theta+\xi)  &\geq h_\sigma(x_\theta) - C|\xi| \geq -H - C|\xi|.
  \end{align}

  We let $\eta_\theta$ be a continuous unit normal vector field to the $g$-segment $\{x_\theta\}$. Fix $\delta > 0$ so that $x_{-1/2}+\delta \eta_{-1/2}$ and $x_{1/2}+\delta \eta_{1/2}$ lie in $\Omega$. For $\epsilon \in [0,\delta]$  let $\{x^\epsilon_\theta\}_{\theta \in [-1/2,1/2]}$ be the $g$-segment with respect to $y_0,z_0$ joining $x_{-1/2}+\epsilon \eta_{-1/2}$ to $x_{1/2}+\epsilon \eta_{1/2}$.  Using Lemma \ref{lem:diffseg} for $\theta \in [-1/4,1/4]$ combined with \eqref{eq:lipest1} and \eqref{eq:lipest2} implies
  \begin{align}
  -C(\epsilon+H) \leq \frac{d}{d\theta}h_{\sigma}(x_\theta^\epsilon) \leq C(\epsilon+H). \label{eq:diffineq2}
  \end{align}
  Here we have used that $|x_\theta - x_\theta^\epsilon| < C \epsilon$ for a Lipschitz constant independent of $\theta$.   

  This implies
  \begin{align}
\label{eq:tosub1}    \int_{-1/4}^{1/4}\frac{d^2}{d\theta^2}h_{\sigma}(x_\theta^\epsilon) \ d \theta \leq C(\epsilon+H),
  \end{align}
  and we come back to this in a moment. For now note that, since
  \[ \det [D^2u - A(\cdot,u,Du)]\geq  c\inf\det E > 0,\] for any two orthogonal unit vectors $\xi,\eta$
 \[ [D_{\xi\xi}u-g_{\xi \xi}(x,Y_u(x),Z_u(x))][D_{\eta\eta}u-g_{\eta \eta}(x,Y_u(x),Z_u(x))] \geq C.\]
 In particular for $\eta_\theta^\epsilon$ a choice of unit normal vector field orthogonal to $\dot{x}_\theta^\epsilon$ (continuous in $\theta,\epsilon$) we have
 \begin{align*}
   C^{-1}&[D_{\dot{x}_\theta^\epsilon\dot{x}_\theta^\epsilon}u-g_{\dot{x}_\theta^\epsilon \dot{x}_\theta^\epsilon}(x,Y_u(x),Z_u(x))] \\
   &\geq |\dot{x}_\theta^\epsilon|^{2} [D_{\eta_\theta^\epsilon\eta_\theta^\epsilon}u-g_{\eta_\theta^\epsilon \eta_\theta^\epsilon}(x,Y_u(x),Z_u(x))]^{-1}.
 \end{align*}
  Employing this and \eqref{eq:diffineq2} in Lemma \ref{lem:diffineq}  gives
  \begin{align}
\label{eq:tosub2}   \frac{d^2}{d\theta^2}h_{\sigma}(x_\theta^\epsilon) &\geq C|\dot{x}_\theta^\epsilon|^2\big([D_{\eta_\theta^\epsilon \eta_\theta^\epsilon}u(x_\theta^\epsilon)-g_{\eta_\theta^\epsilon \eta_\theta^\epsilon}(x_\theta^\epsilon,Y_u(x_\theta^\epsilon),Z_u(x_\theta^\epsilon))]\big)^{-1}\\&\quad\quad - C(\epsilon+H),\nonumber
  \end{align}
  where initially this holds for $h_0$, and thus for $h_\sigma$. Note that by \eqref{eq:A3equiv} the $D_p^2A$ term in Lemma \ref{lem:diffineq} is bounded below by $-C|h'|$ and subsequently controlled by \eqref{eq:diffineq2}.

  Substituting \eqref{eq:tosub2} into \eqref{eq:tosub1} we obtain
  \begin{align*}
    \int_{-1/4}^{1/4} |\dot{x}_\theta^\epsilon|^2\big([D_{ij}u(x_\theta^\epsilon)-g_{ij}(x_\theta^\epsilon)](\eta_\theta^\epsilon)_i (\eta_\theta^\epsilon)_j\big)^{-1} \ d \theta \leq C(\epsilon+H),
  \end{align*}
 where we omit that $g$ is evaluated at $(x_\theta^\epsilon,Y_u(x_\theta^\epsilon),Z_u(x_\theta^\epsilon))$. An application of Jensen's inequality implies
  \begin{align}
   \label{eq:inttobound} &\int_0^\delta \int_{-1/4}^{1/4} |\dot{x}_\theta^\epsilon|^{-2}[D_{ij}u(x_\theta^\epsilon)-g_{ij}(x_\theta^\epsilon)](\eta_\theta^\epsilon)_i (\eta_\theta^\epsilon)_j \ d \theta \ d\epsilon \\\quad&\geq C \int_0^\delta \Big(\int_{-1/4}^{1/4}  |\dot{x}_\theta^\epsilon|^2\big([D_{ij}u(x_\theta^\epsilon)-g_{ij}(x_\theta^\epsilon)](\eta_\theta^\epsilon)_i (\eta_\theta^\epsilon)_j\big)^{-1}\ d \theta \Big)^{-1} \ d\epsilon \nonumber \\
\label{eq:invint}    &\quad \geq  \int_0^\delta \frac{C}{\epsilon+H} \ d \epsilon. 
  \end{align}
This is the crux of the proof complete: the only way for the final integral to be bounded is if $H$ is bounded away from 0. We're left to show the integral \eqref{eq:inttobound} is bounded in terms of the allowed quantities, and approximate when $u$ is not $C^2$.

   To bound \eqref{eq:inttobound} use that $\det E \neq 0$ implies $|\dot{x}_\theta^\epsilon|$ is bounded below by a positive constant depending on $|x_1-x_0|$ and $g$. This gives the estimate
  \begin{align*}
    &\int_0^\delta \int_{-1/4}^{1/4} |\dot{x}_\theta^\epsilon|^{-2}[D_{ij}u(x_\theta^\epsilon)-g_{ij}(x_\theta^\epsilon)](\eta_\theta^\epsilon)_i (\eta_\theta^\epsilon)_j \ d \theta \ d\epsilon\\
   & \leq C  \int_0^\delta \int_{-1/4}^{1/4}  [D_{ij}u(x_\theta^\epsilon)-g_{ij}(x_\theta^\epsilon)](\eta_\theta^\epsilon)_i (\eta_\theta^\epsilon)_j \ d \theta \ d\epsilon\\
   & \leq  C  \int_0^\delta \int_{-1/4}^{1/4} \sum_i D_{ii}u(x_\theta^\epsilon)-g_{ii}(x_\theta^\epsilon) \ d \theta \ d \epsilon.
  \end{align*}
  The final line is obtained using positivity of $D_{ij}u-g_{ij}$ and is bounded in terms of $\Vert g \Vert_{C^2}$ and $\sup |Du|$ (compute the integral in the Cartesian coordinates and note the Jacobian for this transformation is bounded). Thus returning to \eqref{eq:inttobound} and \eqref{eq:invint} we obtain $H > C$ where $C$ depends on the stated quantities. 

  \textit{Step 2: Convexity estimates for Alexandrov solutions via a barrier argument}\\
  We extend to Alexandrov solutions via a barrier argument. Suppose $u$ is an Alexandrov solution of \eqref{eq:alekssoln} that is not strictly convex. There is a support $g(\cdot,y_0,z_0)$ touching at two points $x_1,x_{-1}$. Using \cite[Lemma 2.3]{Trudinger14} we also have $u \equiv g(\cdot,y_0,z_0)$ along the $g$-segment joining these points (with respect to $y_0,z_0$). Balls with sufficiently small radius are $g$-convex. This follows because, as noted in \cite[\S 2.2]{LiuTrudinger16}, $g$-convexity requires the boundary curvatures minus a function depending only on $\Vert g \Vert_{C^3}$ are positive. Thus we  assume $x_1,x_{-1}$ are sufficiently close to ensure that $B$, the ball with radius $|x_{1}-x_{-1}|/2$ and centre $(x_1+x_{-1})/2$ is $g$-convex with respect to $Y_u(\Omega)\times Z_u(\Omega)$. Let $\epsilon > 0$ be given, and let $u_h$ be the mollification of $u$ with $h$ taken small enough to ensure $|u-u_h| < \epsilon/2$ on $\partial B.$ The Dirichlet theory for GJE, \cite[Lemma 4.6]{Trudinger14}, yields a $C^3$ solution of
  \begin{align*}
    \det DY(\cdot,v_h,Dv_h) = c/2 \text{ in }B,\\
    v_h = u_h+\epsilon \text{ on } \partial B,
  \end{align*}
  satisfying an estimate $|Dv_h| \leq K$ where $K$ depends on the local Lipschitz constant of $u$. Since $v_h \geq u$ on $\partial B$ the comparison principle, (\cite[Lemma 4.4]{Trudinger14}) implies  $v_h \geq u$ in $B$. Thus we can apply our previous argument and obtain strict $g$-convexity of $v_h$ provided we note \eqref{eq:sigest1} and \eqref{eq:sigest2} are satisfied for $\sigma = 2\epsilon$. Hence at $x_{\theta_\epsilon}$ a point on the $g$-segment where the infimum defining $H$ is obtained we have 
  \[ u(x_{\theta_\epsilon}) - g(x_{\theta_\epsilon},y_0,z_0) - 2\epsilon \leq v_h(x_{\theta_\epsilon}) - g(x_{\theta_\epsilon},y_0,z_0) - 2\epsilon \leq -H < 0.\]
  As $\epsilon \rightarrow 0$ we contradict that $g(\cdot,y_0,z_0)$ supports $u$.  
 \end{proof}

\section{$C^1$ regularity}
\label{sec:c1}
 For a $g$-convex function defined on $\Omega$ with $Y_u(\Omega) = \Omega^*$ its $g^*$\textit{-transform} is the function defined on $\Omega^*$ by
 \[ v(y):= \sup_{x \in \Omega}g^*(x,y,u(x)). \]
 We list a few essential properties. 
Let us suppose $y_0 \in Y_u(x_0)$. This means
  \begin{align*}
   u(x) \geq g(x,y_0,g^*(x_0,y_0,u(x_0))),
  \end{align*}
  taking $g^*(x,y_0,\cdot)$ of both sides and using $g^*_u<0$ yields
  \begin{align*}
    g^*(x,y_0,u(x)) \leq g^*(x_0,y_0,u(x_0)),
  \end{align*}
  so that $v(y_0) = g^*(x_0,y_0,u(x_0))$. The definition of $v$ implies for other $y\in\Omega^*$ we have $v(y) \geq g^*(x_0,y,u(x_0))$. Thus $g^*(x_0,\cdot,u(x_0))$ is a $g^*$-support at $y_0$. Which is to say $x_0 \in X_v(y_0)$.

We use this as follows. Suppose in addition $u$ satisfies that for all $E \subset \Omega$
\begin{align*}
  |Y_u(E)| \leq c^{-1}|E|.
\end{align*}
Take $A \subset \Omega^*$ and let $E_u$ denote the measure $0$ set of points where $u$ is not differentiable. Necessarily $A\setminus Y_u(E_u) = Y_u(E)$ for  some $E\subset \Omega$. Our above reasoning implies $E \subset X_v(Y_u(E))$. Hence 
\begin{align}
 \label{eq:dualineq} |X_v(A)|\geq |X_v(A\setminus Yu(E_u))| \geq |E| \geq c|A|.
\end{align}

Corollary \ref{cor:c1} follows: Let $u$ be the function given in Corollary \ref{cor:c1} and $v$ its $g^*$ transform defined on $\Omega^*$. Theorem \ref{thm:main} holds in the dual form, that is, provided the relevant hypothesis are changed to their starred equivalents, Theorem \ref{thm:main} implies strict $g^*$-convexity. Thus the hypothesis of Corollary \ref{cor:c1} along with \eqref{eq:dualineq} allow us to conclude $v$ is strictly $g^*$-convex. 

  Suppose for a contradiction $u$ is not $C^1$. Then for some $x_0$ the set $Y_u(x_0)$ contains two distinct points, say $y_0,y_1$. Our above working implies $g^*(x_0,\cdot,u(x_0))$ is a support touching at $y_0,y_1$. This contradicts strict $g^*$-convexity and proves the corollary.  

  \appendix
  \section{Proof of main lemma}
  \label{app:proof}
  In this appendix we provide the proof of Lemma \ref{lem:diffineq}. 
  \begin{proof}
  We first compute a differentiation formula for second derivatives along $g$-segments. We suppose
  \begin{equation}
    \label{eq:gseglemdef} \frac{g_y}{g_z}(x_\theta,y_0,z_0) = \theta q_1+(1-\theta)q_0,
  \end{equation}
and set $q = q_1-q_0$. We begin with a formula for first derivatives. Since
  \begin{equation}
    \label{eq:3chain}
     \frac{d}{d\theta} = (\dot{x_\theta})_iD_{x_i}
  \end{equation}
  we need to compute $(\dot{x_\theta})_i$. Differentiate \eqref{eq:gseglemdef} with respect to $\theta$ and obtain\footnote{We use the convention that subscripts before the comma denote differentiation with respect to $x$, and subscripts after the comma (which are not $z$) denote differentiation with respect to $y$.}
  \[\left[\frac{g_{i,m}}{g_z}-\frac{g_{i,z}g_{,m}}{g_z^2}\right](\dot{x_\theta})_i = q_m,\]
  from which it follows that
  \[ (\dot{x_\theta})_i = g_zE^{m,i}q_m,\]
  where $E^{m,i}$ denotes the m,i$^{\text{th}}$ entry of $E^{-1}$. Thus \eqref{eq:3chain} becomes
  \begin{equation}
    \label{eq:3qderiv}
     \frac{d}{d\theta} = g_zE^{m,i}q_mD_{x_i}.
  \end{equation}
  Using this expression to compute second derivatives we have
  \begin{align*}
  \frac{d^2}{d\theta^2}&= g_zE^{n,j}D_{x_j}(g_zE^{m,i}D_{x_i})q_mq_n\\
                  &= g_{z}^2E^{n,j}E^{m,i}q_mq_nD_{x_ix_j} +g_z^2q_mq_nE^{n,j}D_{x_j}(E^{m,i})D_{x_i}\\
    &\quad\quad+ g_zg_{j,z}E^{n,j}E^{m,i}q_mq_nD_{x_i}.
  \end{align*}
  The formula for differentiating an inverse yields
  \begin{align}
      \label{eq:3tosub}   \frac{d^2}{d\theta^2} = (\dot{x_\theta})_i(\dot{x_\theta})_j&D_{x_ix_j} -g_z^2q_mq_nE^{n,j}E^{m,a}D_{x_j}(E_{ab})E^{b,i}D_{x_i}\\&+ g_zg_{j,z}E^{n,j}E^{m,i}q_mq_nD_{x_i}.\nonumber
  \end{align}
  Now compute
  \begin{align}
    D_{x_j}(E_{ab}) &= D_{x_j}\left[g_{a,b}-\frac{g_{a,z}g_{,b}}{g_z}\right] \nonumber\\
    &= g_{aj,b}-\frac{g_{aj,z}g_{,b}}{g_z}-\frac{g_{a,z}g_{j,b}}{g_z}+\frac{g_{j,z}g_{a,z}g_{,b}}{g_z^2}\nonumber\\
    &= -\frac{g_{a,z}}{g_z}E_{jb} +E_{l,b}D_{p_l}g_{aj} \label{eq:3invderiv}.
  \end{align}
  Here we have used that
  \[ E_{l,b}D_{p_l}g_{aj}(\cdot,Y(\cdot,u,p),Z(\cdot,u,p)) = g_{aj,b}-\frac{g_{aj,z}g_{,b}}{g_z}, \]
  which follows by computing $D_{p_l}g_{aj}$, differentiating \eqref{eq:yzdef1} with respect to $p$ to express  $Z_p$ in terms of $Y_p$, and employing $E^{i,j} = D_{p_j}Y^i$ (which is obtained via calculations similar to those for \eqref{eq:yallderiv}).

   Substitute \eqref{eq:3invderiv} into \eqref{eq:3tosub} to obtain
  \begin{align*}
   \frac{d^2}{d\theta^2} &= (\dot{x_\theta})_i(\dot{x_\theta})_j D_{x_ix_j} -g_z^2q_mq_nE^{n,j}E^{m,a}E_{l,b}D_{p_l}g_{aj}E^{b,i}D_{x_i}\\
               &\quad\quad+[g_zg_{a,z}E^{n,j}E^{m,a}E_{j,b}E^{b,i}D_{x_i+}g_zg_{j,z}E^{n,j}E^{m,i}D_{x_i}]q_mq_n\\
    &= (\dot{x_\theta})_i(\dot{x_\theta})_j D_{x_ix_j} -g_z^2q_mq_nE^{n,j}E^{m,a}D_{p_i}g_{aj}D_{x_i}\\
               &\quad\quad+[g_zg_{a,z}E^{n,i}E^{m,a}D_{x_i+}g_zg_{j,z}E^{n,j}E^{m,i}D_{x_i}]q_mq_n\\
    &= (\dot{x_\theta})_i(\dot{x_\theta})_j(D_{x_i,x_j}-D_{p_k}g_{ij}D_{x_k})\\
               &\quad\quad+g_{j,z}\Big(E^{m,j}q_m\frac{d}{d\theta}+E^{n,j}q_n\frac{d}{d\theta}\Big),
  \end{align*}
  where in the last equality we swapped  the dummy indices $i$ and $a$ on the second term to allow us to collect like terms and also used \eqref{eq:3qderiv}.

  Now let's use this identity to compute $h''(\theta)$. We have 
  \begin{align*}
    h''(\theta) &= \big[D_{ij}u(x_{\theta})-g_{ij}(x_{\theta},y_0,z_0)\\
           &\quad-D_{p_k}g_{ij}(x_{\theta},y_0,z_0)(D_ku(x_\theta)-D_kg(x_{\theta},y_0,z_0))\big](\dot{x_\theta})_i(\dot{x_\theta})_j \\
           &\quad\quad+g_{j,z}(E^{m,j}q_mh'+E^{n,j}q_nh'). 
  \end{align*}
  Terms on the final line are bounded below by $-K|h'(\theta)|$. Now after adding and subtracting $g_{ij}(x_{\theta},y,z)$ for $y=Y_u(x_\theta),z=Z_u(x_\theta)$ we have
        \begin{align*}
          h''(\theta) &\geq  \big[D_{ij}u(x_{\theta}) -g_{ij}(x_{\theta},y,z)\big](\dot{x_\theta})_i(\dot{x_\theta})_j +  [g_{ij}(x_{\theta},y,z)-g_{ij}(x_{\theta},y_0,z_0)\\
                 &\quad-D_{p_k}g_{ij}(x_{\theta},y_0,z_0)(D_ku(x_\theta)-D_kg(x_{\theta},y_0,z_0))\big](\dot{x_\theta})_i(\dot{x_\theta})_j \\
           &\quad\quad-K|h'(\theta)|.
        \end{align*}
        Set $u_0 = g(x_{\theta},y_0,z_0), \ u_1=u(x_\theta)$ $p_0 = g_x(x_{\theta},y_0,z_0),$ and $ p_1 = Du(x_\theta)$. Then rewriting in terms of the matrix $A$ we have
\begin{align*}
  h''(\theta) &\geq  \big[D_{ij}u(x_{\theta}) -g_{ij}(x_{\theta},y,z)\big](\dot{x_\theta})_i(\dot{x_\theta})_j  +\big[A_{ij}(x_{\theta},u_1,p_1)-A_{ij}(x_{\theta},u_0,p_0)\\
         &\quad-D_{p_k}A_{ij}(x_{\theta},u_0,p_0)(p_1-p_0)\big](\dot{x_\theta})_i(\dot{x_\theta})_j -K|h'(\theta)|\\
         &= \big[D_{ij}u(x_{\theta}) -g_{ij}(x_{\theta},y,z)\big](\dot{x_\theta})_i(\dot{x_\theta})_j  + A_{ij,u}(x_\theta,u_\tau,p)(u_1-u_0)(\dot{x_\theta})_i(\dot{x_\theta})_j\\
         &\quad+\big[A_{ij}(x_{\theta},u_0,p_1)-A_{ij}(x_{\theta},u_0,p_0)\\
         &\quad\quad-D_{p_k}A_{ij}(x_{\theta},u_0,p_0)(p_1-p_0)\big](\dot{x_\theta})_i(\dot{x_\theta})_j-K|h'(\theta)|
        \end{align*} 
        Here $u_\tau = \tau u +(1-\tau)u_0$ results from a Taylor series. Another Taylor series for $f(t):= A_{ij}(x_{\theta},u_0,t p_1+(1-t)p_0)$ and we obtain
        \begin{align*}
           h''(\theta) &\geq  \big[D_{ij}u(x_{\theta}) -g_{ij}(x_{\theta},y,z)\big](\dot{x_\theta})_i(\dot{x_\theta})_j  + A_{ij,u}(u_1-u_0)(\dot{x_\theta})_i(\dot{x_\theta})_j\\ &-K|h'(\theta)| + D_{p_kp_l}A_{ij}(x_{\theta},u_0,p_t)(\dot{x_\theta})_i(\dot{x_\theta})_j(p_1-p_0)_k(p_1-p_0)_l. 
        \end{align*}
        This is the desired formula.
\end{proof}

\bibliographystyle{plain}
\bibliography{bib201016C1Regularity2DRevisited}

\end{document}